\documentclass{scrartcl}
\usepackage[utf8]{inputenc}
\usepackage[T1]{fontenc}

\usepackage{amsmath}
\usepackage{amssymb}

\newcommand{\R}{\mathbb{R}}

\renewcommand{\phi}{\varphi}

\DeclareMathOperator*{\argmin}{arg\,min}
\DeclareMathOperator{\sgn}{sgn}

\newcommand{\prox}[3][]{\operatorname{prox}^{#1}_{#2}\left(#3 \right)}

\usepackage{amsthm}
\theoremstyle{plain}
\newtheorem{theorem}{Theorem}[section]
\newtheorem{corollary}{Corollary}[section]
\newtheorem{lemma}{Lemma}[section]

\theoremstyle{definition}
\newtheorem{definition}{Definition}[section]

\theoremstyle{remark}
\newtheorem{remark}{Remark}[section]

\usepackage{algorithm}
\usepackage{algorithmic}

\usepackage{amsfonts}

\usepackage{enumerate}
\usepackage{enumitem}

\usepackage{breqn}

\newcommand*\diff{\mathop{}\!\mathrm{d}}

\usepackage{mathtools}

\usepackage{graphicx}
\usepackage{caption,subcaption}

\newcommand{\dist}{\textup{dist}}
\newcommand{\cO}{\mathcal{O}}
\newcommand{\tcO}{\widetilde{\mathcal{O}}}
\newcommand{\eps}{\epsilon}

\usepackage{mathtools}
\mathtoolsset{showonlyrefs}

\title{Variable Smoothing for Weakly Convex Composite Functions}
\author{Axel B\"ohm\footnote{Faculty of Mathematics, University of Vienna, Oskar-Morgenstern-Platz 1, 1090 Vienna, Austria. e-mail: axel.boehm@univie.ac.at. Research supported by the doctoral programme \textit{Vienna Graduate School on Computational Optimization (VGSCO)},
    FWF (Austrian Science Fund), project W 1260.} \and Stephen J. Wright\footnote{Computer Sciences Department and Wisconsin Institute for Discovery, University of Wisconsin-Madison. e-mail: swright@cs.wisc.edu. Research supported by  NSF Awards 1628384, 1634597, and 1740707; Subcontract 8F-30039 from Argonne National Laboratory; and Award N660011824020 from the DARPA Lagrange Program.}}

\date{\today}
\begin{document}

\maketitle
\begin{abstract}%
  We study minimization of a structured objective function, being the
  sum of a smooth function and a composition of a weakly convex
  function with a linear operator. Applications include image
  reconstruction problems with regularizers that introduce less bias
  than the standard convex regularizers. We develop a variable
  smoothing algorithm, based on the Moreau envelope with a decreasing
  sequence of smoothing parameters, and prove a complexity of
  $\cO(\eps^{-3})$ to achieve an $\eps$-approximate solution. This
  bound interpolates between the $\cO(\eps^{-2})$ bound for the smooth
  case and the $\cO(\eps^{-4})$ bound for the subgradient method. Our
  complexity bound is in line with other works that deal with
  structured nonsmoothness of weakly convex functions.
\end{abstract}

\section{Introduction}%
\label{sec:introduction}
We study minimization of the sum of a smooth function and a
nonsmooth, weakly convex function composed with a linear
operator. The case in which the nonsmooth regularizer is convex has
been studied extensively; see~\cite{pdhg,ROF-TV-denoising}. Weakly
convex functions (which can be expressed as the difference between a
convex function and a quadratic) share some properties with convex
functions but include many interesting nonconvex cases, as we
discuss in Section~\ref{sub:weakly_convex_regularizers}. For
example, any smooth function with a uniformly Lipschitz continuous
gradient is a weakly convex function.

Our approach makes use of a smooth approximation of the
weakly convex function known as the \emph{Moreau envelope},
parametrized by a positive scalar $\mu$.
Since evaluation of the gradient of the Moreau envelope is obtained
by applying a proximal operator to the function, our method is
suitable for problems where this proximal operator can be evaluated
at reasonable cost. Our method requires $\mathcal{O}(\epsilon^{-3})$
iterations to obtain an $\epsilon$-approximate stationary
point.

The remainder of the paper is organized as follows.
Section~\ref{sec:related-problems} is concerned with other problem
formulations related to ours and describes specific problems with
weakly convex regularizers. In Section~\ref{sec:prelim} we give the
necessary mathematical preliminaries including a detailed discussion
about the notion of stationarity we use.  Section~\ref{sec:main}
describes our approach and its convergence properties. In
Section~\ref{sec:prox_grad} we highlight the difference between the
variable smoothing technique and a simple proximal-gradient
approach, for the case in which the linear operator is not present
in the weakly smooth term.

\section{Problem Class and Algorithmic Approach}%
\label{sec:related-problems}

The problem we address in this paper has the form
\begin{equation}
  \label{eq:hg}
  \min_{x \in \R^d} \, F(x) := h(x) + g(Ax),
\end{equation}
for a smooth function $h:\R^d \to \R$, a weakly convex function $g:\R^n
\to \R$ (generally nonsmooth) and a matrix $A\in\R^{n \times d}$.
For some $\rho \ge 0$, we say that
\begin{equation} \label{eq:weakly_convex}
 \mbox{$g:\R^n \to \overline{\R}$ is \em{$\rho$-weakly convex} if} \;\; g
   + \frac{\rho}{2} \lVert \cdot \rVert^2 \; \mbox{ is convex.}
\end{equation}
When $g$ is a smooth function with a uniformly Lipschitz continuous
gradient, with Lipschitz constant $L$, then $g$ is weakly
  convex with $\rho=L$. Other interesting weakly convex functions are
discussed in Section~\ref{sub:weakly_convex_regularizers}.

The \emph{Moreau envelope} $g_{\mu}$ is a smooth approximation of $g$,
parametrized by a positive scalar $\mu$.  The Moreau envelope and
the closely related proximal operator are defined as follows.
\begin{definition}%
  \label{def:moreau_envelope}
  For a proper, $\rho$-weakly convex and lower semicontinuous function
  $g: \R^n \rightarrow \overline{\R}$,  the Moreau envelope of $g$ with the
  parameter $\mu \in (0,\rho^{-1}[$ is the function from $\R^n$
  to $\R$ defined by
  \begin{equation} \label{eq:moreau}
    g_{\mu}(y) := \inf_{z \in \R^n} \left\{ g(z) +
    \frac{1}{2\mu}\lVert z-y \rVert^2\right\}.
  \end{equation}
  The proximal operator of the function $\mu g$ is the $\argmin$ of
  the right-hand side in this definition, that is,
  \[
    \prox{\mu g}{y} := \argmin_{z \in \R^n}\left\{ g(z) +
    \frac{1}{2\mu}\lVert z-y \rVert^2\right\} =
    \argmin_{z \in \R^n}\left\{ \mu g(z) +
    \frac{1}{2}\lVert z-y \rVert^2\right\}.
    \]
\end{definition}
Note that $\prox{\mu g}{y}$ is defined \emph{uniquely} by this
formula, because the function being minimized is strongly convex.
We describe in Lemma~\ref{lem:grad_smooth_is_prox} below the
relationship between $\nabla g_{\mu}(y)$ and $\prox{\mu g}{y}$,
which is key to our algorithm.

Steps of our algorithm have the form
\begin{equation}
  x \leftarrow x - \gamma \nabla (h + g_{\mu} \circ A)(x),
\end{equation}
for some steplength $\gamma$. Accelerated versions of these approaches
have been proposed for convex problems
in~\cite{me-variable-smoothing,cevher-smoothing-acc,bot15variable-smoothing}.
The use of acceleration makes the analysis more complicated than for
the gradient case; see~\cite{fista,chambolle_dossal}.

\subsection{Composite Problems}

We discuss several instances of problems of the form~\eqref{eq:hg}.

\paragraph{Regularization with $\Vert \cdot \Vert_1$ (LASSO).}
Functions that are ``sharp'' around zero have a long history as
regularizers that induce sparsity in the solution vector
$x$. Foremost among such functions is the vector norm $\Vert \cdot
\Vert_1$, which is used for example in sparse least-squares
regression (also known as LASSO~\cite{Tib96}):
\begin{equation} \label{eq:lasso}
  \min_x \, \frac12 \Vert Bx - b \Vert^2 + \Vert x \Vert_1.
\end{equation}
This formulation is convex and forms a special case of~\eqref{eq:hg}
in which $A$ is given by the identity. Regularization with the norm
$\Vert \cdot \Vert_1$ is used also in logistic
regression~\cite{ShiWWLKK06}.

\paragraph{Other Convex Regularizers.}
The case of problems~\eqref{eq:hg} in which $g$ is nonsmooth and
\textit{convex} (with possible smooth and/or nonsmooth additive terms)
has received a great deal of attention in the literature on convex
optimization and applications; see for
example~\cite{cevher-smoothing-acc,pdhg,vu,bot2015primaldualgap,bot2014primaldualstrong}.
The most notable applications are found in inverse problems involving
images. In particular, discrete (an)isotropic \emph{Total Variation
  (TV)} denoising has the form
\begin{equation} \label{eq:tv.dn}
  \min_x \, \frac12 \Vert x-b \Vert^2 + \Vert \nabla x \Vert_1,
\end{equation}
where $b$ is the observed (noisy) image and $\nabla$ denotes the
discretized gradient in two or three dimensions. TV deblurring
problems have the form
\begin{equation} \label{eq:tv.db}
  \min_x \, \frac12 \Vert Bx-b \Vert^2 + \Vert \nabla x \Vert_1,
\end{equation}
where $B$ is the blurring operator; see~\cite{pdhg,spdhg}.

Other examples of convex problems of the form~\eqref{eq:hg} include
generalized convex feasiblity~\cite{cevher-smoothing-acc} and support
vector machine classification~\cite{bot15variable-smoothing}. A
typical formulation of the latter problem has $h(x) = (\lambda/2)
\|x\|^2$ and $g(Ax) = \sum_{i=1}^n \phi(y_i a_i^T x)$, where $\phi(s) =
\max\{-s,0\}$ is the hinge loss and the rows of $A$ are $y_i a_i^T$,
$i=1,2,\dotsc,n$, where $(y_i,a_i) \in \{-1,1\} \times \R^d$ are the
training points and their labels.

\paragraph{Weakly Convex Regularizers.}%
\label{sub:weakly_convex_regularizers}

The use of the $\ell_1$ regularizer in \eqref{eq:lasso} tends
to depress the magnitude of nonzero elements of the solution,
resulting in {\em bias}. This phenomenon is a consequence of the fact
that the proximal operator of the $1$-norm, often called the
\emph{soft thresholding operator}, does not approach the identity
even for large arguments. For this reason, nonconvex
alternatives to $\Vert \cdot \Vert_1$ are often used to reduce bias.
These include $\ell_p$-norms (with $0<p<1$) which are not weakly
convex, and the several weakly convex regularizers, which we now
describe.
The \textit{minimax concave penalty (MCP)}, introduced in~\cite{mcp}
and used in~\cite{log-regression-mcp,calculus-KL}, is a family of
functions $r_{\lambda,\theta}:\R\to\R_+$ involving two positive
parameters $\lambda$ and $\theta$, and defined by
\begin{equation}
  \label{eq:mcp}
  r_{\lambda,\theta}(x) :=
  \begin{cases}
    \lambda\vert x \vert - \frac{x^2}{2 \theta},   &   \vert x \vert \le \theta\lambda,\\
    \frac{\theta \lambda^2}{2},            &   \text{otherwise}.
  \end{cases}
\end{equation}
(Note that this function satisfies the definition of $\rho$-weak
convexity with $\rho = \theta^{-1}$.) The proximal operator of this
function (called \emph{firm threshold} in~\cite{sparse-signal-weak})
can be written in the following closed form when $\theta>\gamma$:
\begin{equation}
  \prox{\gamma r_{\lambda,\theta}}{x} =
  \begin{cases}
    0,                           & \vert x \vert < \gamma \lambda, \\
    \frac{x- \lambda \gamma \sgn(x)}{ 1-({\gamma}/{\theta})}, & \gamma \lambda \le \vert x \vert \le \theta \lambda, \\
    x ,                           & \vert x \vert > \theta \lambda .
  \end{cases}
\end{equation}
The \textit{fractional penalty function}
(cf.~\cite{weaklyconvex_regularizer,calculus-KL}) $\phi_a:\R\to\R_+$
(for parameter $a>0$) is
\begin{equation}
  \phi_a(x) := \frac{\vert x \vert}{1 + a \vert x \vert /2 }.
\end{equation}
The \textit{smoothly clipped absolute deviation (SCAD)}~\cite{scad}
(cf.~\cite{calculus-KL}) is defined for parameters $\lambda>0$ and
$\theta>2$ as follows:
\begin{equation}
  \label{eq:scad}
  r_{\lambda,\theta}(x) =
  \begin{cases}
    \lambda\vert x \vert ,                                &   \vert x \vert \le \lambda, \\
    \frac{-x^2+2\theta\lambda \vert x \vert - \lambda^2}{2(\theta-1)},      &   \lambda < \vert x \vert \le \theta \lambda, \\
    \frac{(\theta+1)\lambda^2}{2},                      &  \vert x \vert > \theta \lambda.
  \end{cases}
\end{equation}
(This function is ${(\theta-1)}^{-1}$-weakly convex.)

Since these functions approach (or attain) a finite value as their
argument grows in magnitude, they do not introduce as much bias in the
solution as does the $\ell_1$ norm, and their proximal operators
approach the identity for large arguments.

The regularizers of this section, and the convex $\Vert \cdot \Vert$
regularizer, have been used mostly in the simple additive setting
\begin{equation}
  \label{eq:weakly_add}
  \min_{x \in \R^d} \, h(x) + g(x)
\end{equation}
for a smooth data fidelity term $h$ and nonsmooth regularizer $g$, for
example in least squares or logistic
regression~\cite{log-regression-mcp} and compressed sensing
(cf.~\cite{sparse-signal-weak}).

\paragraph{Weakly Convex Composite Losses.}%
\label{sub:weakly_convex_losses}

The use of weakly convex functions composed with linear operators has
been explored in the robust statistics literature.  An early instance
is the {\em Tukey biweight} function~\cite{BeaT74}, in which $g(Ax)$
has the form
\begin{equation} \label{eq:tukey}
  g(Ax) = \sum_{i=1}^n \phi(A_{i\cdot}x-b_i), \quad
  \mbox{where} \; \phi(\theta) = \frac{\theta^2}{1+\theta^2},
\end{equation}
where $A_{i\cdot}$ denotes the $i$-th row of $A$.
This function behaves like the usual least-squares loss when
$\theta^2 \ll 1$ but asymptotes at $1$. It is $\rho$-weakly convex
with $\rho=6$.

A slightly different definition of the Tukey biweight function appears
in~\cite[Section~2.1]{Loh17a}. This same reference also mentions
another nonconvex loss function, the {\em Cauchy loss}, which has the
form~\eqref{eq:tukey} except that $\phi$ is defined by
\[
  \phi(\theta) = \frac{\xi^2}{2} \log \left( 1+ \frac{\theta^2}{\xi^2} \right),
\]
for some parameter $\xi$. This function is $\rho$-weakly convex with
$\rho=6$.

\subsection{Complexity Bounds for Weakly Convex Problems}%
\label{sub:weakly_convex_rates}

To put our results in perspective, we provide a review of the
literature on complexity bounds for optimization problems related to
our formulation~\eqref{eq:hg}, including weakly convex functions. In
all cases, these are bounds on the number of iterations required to
find an approximately stationary point, where our measure of
stationarity is based the norm of the gradient of the Moreau envelope
(a smooth proxy).

The best known complexity for black box subgradient optimization for
weakly convex functions is $\cO(\epsilon^{-4})$. This result is proved
for \emph{stochastic} subgradient in~\cite{stoch-weak-sub-k-4}, but as
in the convex case, there is no known improvement in the deterministic
setting.
As in convex optimization, subgradient methods are quite general and
implementable for weakly convex functions. However, when more
structure is present in the function, algorithms that achieve better
complexity can be devised. In particular, when the proximal operator
of the nonsmooth weakly convex function can be calculated
analytically, complexity bounds of $\cO(\epsilon^{-2})$ can be proven
(see Section~\ref{sec:prox_grad}), the same bounds as for steepest
descent methods in the smooth nonconvex case. This means that the
entire difficulty introduced by the nonsmoothness can be mitigated as
long as the nonsmoothness can be  treated by a proximal operator.

For convex optimization problems, bounds of $\cO(\epsilon^{-1})$ can
be obtained for gradient methods on smooth functions and
$\cO(\epsilon^{-1/2})$ for accelerated gradient methods.  These same
bounds can also be obtained for nonsmooth problems provided that the
nonsmooth term is handled by a proximal operator. When the
explicit proximal operator is not available and subgradient methods
have to be used, the complexity reverts to $\cO(\epsilon^{-2})$.

It is possible to keep the $\cO(\epsilon^{-2})$ rate when just a local
model of the weakly convex part is evaluated by a convex operator. The
paper~\cite{paquette_composition_k-2} studies optimization
problems of the type
\begin{equation} \label{eq:hc}
  \min_x \, h(x) + g(c(x))
\end{equation}
where $h$ is convex, proper, and closed; $g$ is convex and Lipschitz
continuous; and $c$ is smooth. (Under these assumptions, the
composition $g\circ c$ is weakly convex.) The  $\cO(\epsilon^{-2})$ bound is
proved for an algorithm in which the (convex) subproblem
\begin{equation}
  \label{eq:subproblem}
  \min_y \, h(y) + g(c(x) + \nabla c(x)(y-x)) + \frac{1}{2 t} \Vert y-x \Vert^2
\end{equation}
is solved explicitly.  In the more realistic case in
which~\eqref{eq:subproblem} must be solved by an iterative procedure,
a bound of $\tcO(\epsilon^{-3})$ is obtained
in~\cite{paquette_composition_k-2}.  (The symbol $\tcO$ hides
logarithmic terms.)

Functions of the form $g(c(x))$ have also been studied
in~\cite{lewis2016proximal} for the case of a smooth nonlinear vector
function $c$ and a prox-regular $g$. This formulation is more general
than those considered in this paper, both in the fact
that all weakly convex functions are prox-regular, and in the
nonlinearity of the inner map $c$. The subproblems
in~\cite{lewis2016proximal} have a form similar
to~\eqref{eq:subproblem}, and while convergence results are proved in
the latter paper, it does not contain rate-of-convergence results or
complexity results.

A different weakly convex structure is explored
in~\cite{smooth-minimax}, in which the weak convexity stems from a
smooth saddle point problem. This paper studies the problem
\begin{equation}
  \min_x \, \max_{y\in Y} \, l(x,y),
\end{equation}
for a compact set $Y \subset \R^m$, where $l(x,\cdot)$ is concave,
$l(\cdot,y)$ is nonconvex, and $l(\cdot,\cdot)$ is smooth.  An
iteration bound of $\tcO(\epsilon^{-3})$ is proved for a method that
uses only gradient evaluations.

In light of the considerations above, the complexity bound of
$\cO(\epsilon^{-3})$ for our algorithm seems almost inevitable. It
interpolates between the setting without structural assumptions about
the nonsmoothness (black box subgradient) and the perfect structural
knowledge of the nonsmoothness (explicit knowledge of the proximal
operator).

In Section~\ref{sec:prox_grad}, we treat the simpler setting in which
the linear operator from~\eqref{eq:hg} is the identity, so that
$F(x) = h(x)+g(x)$. Similar problems have been analyzed before, for
example, in~\cite{log-regression-mcp,sparse-signal-weak}. However, it
is assumed in~\cite{sparse-signal-weak} that convexity in the data
fidelity term $h$ compensates for nonconvexity in the regularizer $g$
such that the overall objective function $F$ remains convex. (We make
no such assumption here.) The paper~\cite{log-regression-mcp} does not
make such restrictive assumptions and proves convergence but not
complexity bounds.

\section{Preliminaries}%
\label{sec:prelim}

The concept of subgradient of a convex function can be adapted to
weakly convex functions via the following definition.
\begin{definition}[Fr\'echet subdifferential]%
  \label{def:frechet_subdifferential}
  Let $g:\R^n \to \overline{\R}$ be a function and $\bar{y}$ a point
  such that $g(\bar{y})$ is finite. Then, the \emph{Fr\'echet
    subdifferential} of $g$ at $\bar{y}$, denoted by $\partial
  g(\bar{y})$, is the set of all vectors $v \in \R^n$ such that
  \begin{equation}
    \label{eq:subgradient}
    g(y) \ge g(\bar{y}) + \langle v, y- \bar{y}\rangle + o(\Vert
    y-\bar{y} \Vert) \quad \text{as $y \to \bar{y}$}.
  \end{equation}
\end{definition}
Modifying the convex case, in which subgradients are the slopes of
linear functions that underestimate $g$ but coincide with it at
$\bar{y}$, Fr\'echet subgradients do so \emph{up to first order}. This
definition makes sense for arbitrary functions, but for lower
semicontinuous $\rho$-weakly convex functions, more can be said. For
example, for this class of function we know that subgradients satisfy
the following stronger version of~\eqref{eq:subgradient}, for all $v
\in \partial g(\bar{y})$,
\begin{equation}
  g(y) \ge g(\bar{y}) + \langle v, y - \bar{y} \rangle - \frac{\rho}{2} \Vert y-\bar{y} \Vert^2, \quad \forall y \in \R^{n}.
\end{equation}
Further, if we assume the weakly convex function to be continuous at a
point $y$, then its subdifferential is nonempty at $y$. Both of these
claims can be verified directly by adding
$\frac{\rho}{2}\Vert \cdot \Vert^{2}$ to $g$ and considering the
convex subdifferential; see~\cite[Lemma~2.1]{stochastic-model-based}.

Another nice property of weakly convex functions is that the
definition of a Moreau envelope (see
Definition~\ref{def:moreau_envelope}) extends without modification to
weakly convex functions, subject only to a restriction on the
parameter $\mu$. The proximal operator of
Definition~\ref{def:moreau_envelope} also extends to this setting, and
this operator and the Moreau envelope fulfil the same identity as in
the convex setting.
\begin{lemma}[{\cite[Corollary 3.4]{moreau-weak-properties}}]%
  \label{lem:grad_smooth_is_prox}
  Let $g: \R^n \to \overline{\R}$ be a proper, $\rho$-weakly convex,
  and lower semicontinuous function, and let $\mu \in (0,\rho^{-1})$.
  Then the Moreau envelope $g_\mu(\cdot)$ is continuously differentiable
  on $\R^n$ with gradient
  \begin{equation}
    \nabla g_\mu(y) = \frac{1}{\mu}\left(y - \prox{\mu g}{y}\right), \quad \mbox{for all $y \in \R^n$}.
  \end{equation}
  This gradient is
  $\max\left\{\mu^{-1}, \frac{\rho}{1-\rho\mu} \right\}$-Lipschitz
  continuous. In particular, 
  a gradient step with respect to the Moreau envelope corresponds to a
  proximal step, that is,
  \begin{equation} \label{eq:ys9}
    y - \mu \nabla g_\mu(y) = \prox{\mu g}{y}, \quad \mbox{for all $y \in \R^n$}.
  \end{equation}
\end{lemma}

Lemma~\ref{lem:grad_smooth_is_prox} not only clarifies the smoothness
of the Moreau envelope, but also gives a way of computing its gradient
via the prox operator. Obviously, a smooth representation whose
gradient could not be computed would be of only limited usefulness
from an algorithmic standpoint. The only difference between the weakly
convex and convex settings is that the Moreau envelope need not be
convex in the former case.

\subsection{Stationarity}%
\label{sub:stationarity}

We say that a point $\bar{x}$ is a stationary point for a function if
the Fr\'echet subdifferential of the function contains $0$ at
$\bar{x}$. The concept of \emph{nearly stationary} is not quite so
straightforward. We motivate our approach by looking first at the
simple additive composite problem, also discussed in
Section~\ref{sec:prox_grad}, which corresponds to setting $A=I$
in~\eqref{eq:hg}, that is,
\begin{equation} \label{eq:hg.add}
  \min_x \, h(x) + g(x).
\end{equation}
Stationarity for~\eqref{eq:hg.add} means that
$ 0 \in \partial( h + g)(\bar{x})$, that is,
$ -\nabla h(\bar{x}) \in \partial g(\bar{x})$.  A natural definition
for $\epsilon$-approximate stationarity would thus be
\begin{equation} \label{eq:near_stat_1}
  \dist (-\nabla h(x), \partial g(x)) \le \epsilon,
\end{equation}
where $\dist$ denotes the distance between two sets and is given
  for a point $x\in\R^d$ and a set ${\cal A}\subset\R^d$ by
  $\dist(x,{\cal A}) := \inf_{y \in {\cal A}} \, \{\Vert x-y \Vert\}$. However, since we are running
gradient descent on the \emph{smoothed} problem, our algorithm will naturally
compute and detect points with that satisfy a threshold condition of the form
\begin{equation}
  \label{eq:algo_near_stationary}
  \lVert \nabla h (x) + \nabla g_\mu(x) \rVert \le \epsilon.
\end{equation}
The next lemma helps to clarify relationship between these two
conditions.
\begin{lemma}%
  \label{lem:grad-moreau-in-partial-prox}
  Let $g: \R^n \to \overline{\R}$ be a proper, $\rho$-weakly convex,
  and lower semicontinuous function; and let $\mu \in (0,\rho^{-1})$.
  Then
 \begin{equation}
    \label{eq:2}
    \nabla g_\mu(x) \in \partial g (\prox{\mu g}{x}).
  \end{equation}
\end{lemma}
\begin{proof}
  From Definition~\ref{def:moreau_envelope}, we have that
  \[
    0 \in \partial g(\prox{\mu g}{x}) + \frac{1}{\mu} (\prox{\mu g}{x} - x),
  \]
  from which the result follows when we use~\eqref{eq:ys9}.
\end{proof}
(This result is proved for the case of $g$ convex in~\cite[Lemma
2.1]{paquette_composition_k-2}, with essentially the same proof.)

This lemma tells us that when~\eqref{eq:algo_near_stationary} holds,
then~\eqref{eq:near_stat_1} is nearly satisfied, except that in the
argument of $\partial g$, $x$ is replaced by $\prox{\mu g}{x}$.  In
general, however, $\prox{\mu g}{x}$ might be arbitrarily far away from
$x$. We can remedy this issue by requiring $g$ to be Lipschitz
continuous also.
\begin{lemma}%
  \label{lem:g-Lipschitz}
  Let $g:\R^n \to \R$ be a $\rho$-weakly convex function
  that is $L_g$-Lipschitz continuous, and let $\mu \in (0,\rho^{-1})$.
  Then the Moreau envelope $g_{\mu}$ is Lipschitz continuous with
  \begin{equation}
    \label{eq:lipschitz}
    \Vert \nabla g_\mu(x) \Vert \le L_g
  \end{equation}
  and
  \begin{equation}
    \label{eq:dist_prox_bounded}
    \Vert x - \prox{\mu g}{x} \Vert \le \mu L_g, \quad \forall x\in\R^{n}.
  \end{equation}
\end{lemma}
\begin{proof}
  Lipschitz continuity is equivalent to bounded
  subgradients~\cite{mordukhovich-variational-analysis}, so
  by~\eqref{eq:2}, we have for all $x\in\R^{n}$
  \begin{equation}
    \Vert \nabla g_\mu(x) \Vert \le \sup \left\{\Vert v \Vert \, : \, v \in \partial g(\prox{\mu g}{x})\right\} \le L_g,
  \end{equation}
  proving~\eqref{eq:lipschitz}. The bound~\eqref{eq:dist_prox_bounded}
  follows immediately when we use the fact that
  $x - \prox{\mu g}{x} = \mu \nabla g_\mu(x)$ from
  Lemma~\ref{lem:grad_smooth_is_prox}.
\end{proof}

When $x\in \R^{n}$ satisfies~\eqref{eq:algo_near_stationary}, $\nabla h$ is
$L_{\nabla h}$-Lipschitz continuous, $g$ is $L_{g}$ Lipschitz
continuous, we have
\begin{alignat}{2}
  \nonumber
    \dist & (-\nabla h(\prox{\mu g}{x}), \partial g(\prox{\mu g}{x})) \;\; \\
  \nonumber
    &\le \Vert \nabla h(\prox{\mu g}{x}) - \nabla h(x) \Vert + \dist(-\nabla h(x), \partial g(\prox{\mu g}{x}))  \\
  \nonumber
    &\le L_{\nabla h}\Vert x- \prox{\mu g}{x} \Vert + \epsilon  \hspace{110pt} \mbox{(from~\eqref{eq:algo_near_stationary} and~\eqref{eq:2})}\\
    &\le L_{\nabla h}L_g\mu + \epsilon \hspace{162pt} \mbox{(from~\eqref{eq:dist_prox_bounded}).}
  \label{eq:prox_is_near_stationary}
\end{alignat}
Thus, if $\mu$ is sufficiently small and $x$
satisfies~\eqref{eq:algo_near_stationary}, then $\prox{\mu g}{x}$ is
near-stationary for~\eqref{eq:hg.add}.

\subsection{Stationarity for the Composite Problem}

It follows immediately from~\eqref{eq:2} in
Lemma~\ref{lem:grad-moreau-in-partial-prox} that for
$\mu \in (0,\rho^{-1})$, we have for all $x\in\R^d$
\begin{equation}
  \label{eq:prox_step_is_gradient_step_moreau_with_A}
  \nabla (g_\mu \circ A)(x) = A^*\nabla g_\mu(Ax) \in A^*\partial
  g(\prox{\mu g}{Ax}).
\end{equation}
Extending the results of the previous section to the case of a general
linear operator $A$ in~\eqref{eq:hg} requires some work. Stationarity
for~\eqref{eq:hg} requires that
$ 0 \in \nabla h(x) + A^*\partial g(Ax)$, so $\epsilon$-near
stationarity requires
\begin{equation} \label{eq:sj9}
  \dist(-\nabla h(x), A^*\partial g(Ax)) \le \epsilon.
\end{equation}
Our method can compute a point $x$ such that
\begin{equation}
  \left\Vert \nabla h(x) + \nabla(g_\mu \circ A)(x) \right\Vert \le \epsilon
\end{equation}
which by~\eqref{eq:prox_step_is_gradient_step_moreau_with_A} implies that
\begin{equation}
  \label{eq:new_stationarity}
  \dist( -\nabla h(x), A^*\partial g(z)) \le \epsilon, \quad
  \mbox{where} \;\; z=\prox{\mu g}{Ax},
\end{equation}
where, provided that $g$ is $L_g$-Lipschitz continuous, we have
\begin{equation} \label{eq:new_stationarity_2}
  \Vert Ax - z \Vert \le  L_g\mu.
\end{equation}
The bound in~\eqref{eq:new_stationarity} measures the criticality,
while the bound in~\eqref{eq:new_stationarity_2} concerns
feasibility. The
bounds~\eqref{eq:new_stationarity},~\eqref{eq:new_stationarity_2}
are not a perfect match with~\eqref{eq:sj9}, since the
subdifferentials of $h$ and $g \circ A$ are evaluated at different
points.

\paragraph{Surjectivity of $\mathbf{A}$.}
When $A$ is surjective, we can perturb the $x$ that
satisfies~\eqref{eq:new_stationarity},~\eqref{eq:new_stationarity_2}
to a nearby point $x^*$ that satisfies a bound of the
form~\eqref{eq:sj9}. Since $z=\prox{\mu g}{Ax}$ is in the range of
$A$,
we can define
\begin{equation} \label{eq:sj1}
  x^* := \argmin_{x'\in \R^{d}} \{\Vert x - x' \Vert^2 \; :  \; Ax' = z\},
\end{equation}
which is given explicitly by
\begin{equation} \label{eq:sj2}
  x^{*} = x - A^{*} {\left( AA^{*} \right)}^{-1}(Ax-z) = x- A^\dagger (Ax-z)
\end{equation}
where $A^\dagger := A^{*}{(AA^{*})}^{-1}$ is the pseudoinverse of $A$.
The operator norm of the pseudoinverse is bounded by the inverse of
the smallest singular value $\sigma_{\min}(A)$ of $A$, so when $g$ is
$L_g$-Lipschitz continuous, we have from~\eqref{eq:new_stationarity_2}
that
\begin{equation}
  \label{eq:use-norm-pseudoinverse}
  \Vert x - x^* \Vert \le  {\sigma_{\min}(A)}^{-1} \|Ax-z\| \le {\sigma_{\min}(A)}^{-1} L_{g}\mu.
\end{equation}
The point $x^{*}$ is approximately stationary in the sense
of~\eqref{eq:sj9}, for $\mu$ sufficiently small, because
\begin{alignat}{2}
  \nonumber
  \dist & (-\nabla h(x^{*}), A^* \partial g(Ax^{*})) && \;\; \\
  \nonumber
  &\le \Vert \nabla h(x^{*}) - \nabla h(x) \Vert + \dist(-\nabla h(x),
  A^* \partial g(z)) && \;\; \mbox{(since $Ax^{*}=z$)} \\
  \nonumber
  &\le L_{\nabla h}\Vert x- x^{*} \Vert + \epsilon && \;\; \mbox{(from~\eqref{eq:new_stationarity})} \\
  &\le L_{\nabla h}{\sigma_{\min}(A)}^{-1}L_g\mu + \epsilon && \;\; \mbox{(from~\eqref{eq:use-norm-pseudoinverse}).}
  \label{eq:xstar_is_near_stationary}
\end{alignat}

By choosing $\mu$ small, $x^{*}$ will be an approximate solution in
the stronger sense~\eqref{eq:sj9} and not just the weaker notion
of~\eqref{eq:new_stationarity},~\eqref{eq:new_stationarity_2}, which
is the case if $A$ is not surjective.

\section{Variable Smoothing}%
\label{sec:main}

We describe our variable smoothing approaches for the
problem~\eqref{eq:hg}, where we assume that $h$ is
$L_{\nabla h}$-smooth; $g$ is possibly nonsmooth, $\rho$-weakly convex,
and $L_g$-Lipschitz continuous; and $A$ is a nonzero linear continuous
operator. For convenience, we define the smoothed approximation
$F_{k}:\R^{d} \to \R$ based on the Moreau envelope with parameter $\mu_k$
as follows:
\begin{equation} \label{eq:Fk}
  F_k(x) := h(x) + g_{\mu_k} (Ax).
\end{equation}
We note from Lemma~\ref{lem:grad_smooth_is_prox} and the chain rule
that
\begin{equation} \label{eq:Fkd}
  \nabla F_k(x) = \nabla h(x) + \frac{1}{\mu_k} A^* (Ax-\prox{\mu_k g}{Ax}).
\end{equation}
The quantity $L_k$ defined by
\begin{equation}%
  \label{eq:Lk}
  L_k := L_{\nabla h} + \Vert A \Vert^2 \max\left\{\mu^{-1}_{k}, \frac{\rho}{1-\rho\mu_{k}} \right\}
\end{equation}
is a Lipschitz constant of the gradient of $\nabla F_k$; see
Lemma~\ref{lem:grad_smooth_is_prox}.  When $\rho \mu_k \le 1/2$, the
maximum in~\eqref{eq:Lk} is achieved by $\mu_k^{-1}$, so in this case
we can define
\begin{equation}%
  \label{eq:Lk2}
  L_k := L_{\nabla h} + \Vert A \Vert^2 / \mu_{k}.
\end{equation}

\subsection{An Elementary Approach}

Our first algorithm takes gradient descent steps on the smoothed
problem, that is,
\begin{equation} \label{eq:Fk.update}
  x_{k+1} = x_k - \gamma_k \nabla F_k(x_k),
\end{equation}
for certain values of the parameter $\mu_k$ and step size $\gamma_k$.
From~\eqref{eq:Fkd}, the formula~\eqref{eq:Fk.update} is equivalent to
\begin{equation}
  x_{k+1} = x_k - \frac{\gamma_k}{\mu_k}A^*(Ax_k - \prox{\mu_k g}{Ax_k}) -
  \gamma_k \nabla h(x_k).
\end{equation}
Our basic algorithm is described next.

\begin{algorithm}[ht!]
  \caption{Variable Smoothing}%
  \label{alg:variable_weakly_smoothing}
  \begin{algorithmic}
    \REQUIRE $x_1 \in \R^d$;
    \FOR{$k=1,2,3,\dotsc$}
    \STATE Set $\mu_k \leftarrow {(2\rho)}^{-1} k^{-1/3}$, define $L_k$ as
    in~\eqref{eq:Lk2}, set $\gamma_k \leftarrow 1/L_k$;
    \STATE Set
$x_{k+1} \leftarrow x_k - \gamma_k \nabla F_k(x_k)$;
    \ENDFOR
  \end{algorithmic}
\end{algorithm}

We now state the convergence result for
Algorithm~\ref{alg:variable_weakly_smoothing}. This result and later
results make use of a quantity
\begin{equation} \label{eq:def.f*}
  F^* := \liminf_{k \to \infty} F_k(x_k),
\end{equation}
which is finite if $F$ is bounded below (and possibly in other
circumstances too). (When $F^* = -\infty$, the claim of the theorem is
vacuously true.) We also make use of the following quantity:
\begin{equation} \label{eq:xjs}
  x^{*}_j:= x_j - A^{\dagger}(Ax_j-\prox{\mu_j g}{Ax_j}).
  \end{equation}

\begin{theorem}%
  \label{thm:variable_weakly_smoothing}
  Suppose that Algorithm~\ref{alg:variable_weakly_smoothing} is
  applied to the problem~\eqref{eq:hg}, where $g$ is $\rho$-weakly
  convex and $\nabla h$ and $g$ are Lipschitz continuous with
  constants $L_{\nabla h}$ and $L_g$, respectively. We have
  \begin{equation}
    \label{eq:thm1.1}
    \begin{aligned}
      \min_{1\le j \le k} \, \dist( -\nabla h(x_j), & A^*\partial g(\prox{\mu_j g}{Ax_j})) \\
      \le & k^{-1/3} \sqrt{L_{\nabla h} + 2\rho \Vert A \Vert^2}\sqrt{F_1(x_1) - F^* + {(2\rho)}^{-1} L_g^2},
    \end{aligned}
  \end{equation}
  where
  \begin{equation}
    \label{eq:thm1.2}
    \Vert Ax_j - \prox{\mu_j g}{Ax_j} \Vert \le j^{-1/3} {(2\rho)}^{-1} L_g,
  \end{equation}
  and $F^*$ is defined as in~\eqref{eq:def.f*}.  If $A$ is also
  surjective, then for $x_k^*$ defined as in~\eqref{eq:xjs},
  we have
  \begin{equation}
    \begin{aligned}
      \min_{1\le j \le k} \,& \dist  ( -\nabla h(x^{*}_j), A^*\partial g(Ax_j^{*})) \\
      \le & k^{-1/3} \bigg(\sqrt{L_{\nabla h} + 2\rho \Vert A \Vert^2}\sqrt{F_1(x_1) - F^* + {(2\rho)}^{-1} L_g^2} + L_{\nabla h} {\sigma_{\min}(A)}^{-1}L_{g}\bigg)
    \end{aligned}
  \end{equation}
  and $\Vert x_{j}-x_{j}^{*} \Vert \le {\sigma_{\min}(A)}^{-1}L_{g}\mu_{j} = {\sigma_{\min}(A)}^{-1}L_{g} {(2\rho)}^{-1} j^{-1/3}$.
\end{theorem}

Before proving this theorem, we state and prove a lemma that relates
the function values of two Moreau envelopes with two different
smoothing parameters. In the convex case, such statements are well
known, but in the nonconvex case this result is novel.
\begin{lemma}%
  \label{lem:relate_function_values}
  Let $g: \R^n \to \overline{\R}$ be a proper, closed, and $\rho$-weakly convex
  function, and let $\mu_2$ and $\mu_1$ be parameters such that $0 <
  \mu_2 \le \mu_1 < \rho^{-1}$. Then, we have
  \begin{equation}
    g_{\mu_2}(y) \le g_{\mu_1}(y) + \frac12 \frac{\mu_1 - \mu_2}{\mu_2}\mu_1\Vert \nabla g_{\mu_1}(y) \Vert^2.
  \end{equation}
  If, in addition, $g$ is $L_g$-Lipschitz continuous, we have
  \begin{equation}
    g_{\mu_2}(y) \le g_{\mu_1}(y) + \frac12 \frac{\mu_1 - \mu_2}{\mu_2}\mu_1 L_g^2.
  \end{equation}
\end{lemma}
\begin{proof}
  By using the definition of the Moreau envelope, together with
  Lemma~\ref{lem:grad_smooth_is_prox}, we obtain
  \begin{equation}
    \begin{aligned}
      g_{\mu_2}(y) &= \min_{u\in\R^n} \, \left\{ g(u) + \frac{1}{2 \mu_2}\Vert y-u \Vert^2\right\} \\
                 &= \min_{u\in\R^n} \, \left\{ g(u) + \frac{1}{2 \mu_1}\Vert y-u \Vert^2 + \frac12\left(\frac{1}{\mu_2} - \frac{1}{\mu_1}\right) \Vert y-u \Vert^2\right\} \\
      &\le g(\prox{\mu_1 g}{y}) + \frac{1}{2 \mu_1}\Vert y-\prox{\mu_1 g}{y} \Vert^2 + \frac12\left(\frac{1}{\mu_2} - \frac{1}{\mu_1}\right) \Vert y-\prox{\mu_1 g}{y} \Vert^2 \\
                 &= g_{\mu_1}(y) + \frac12\left(\frac{\mu_1-\mu_2}{\mu_2}\right) \mu_1 \Vert \nabla g_{\mu_1} (y) \Vert^2,
    \end{aligned}
  \end{equation}
  proving the first claim. The second claim follows immediately
  from~\eqref{eq:lipschitz}.
\end{proof}

\begin{proof}[Proof of Theorem~\ref{thm:variable_weakly_smoothing}]
  Since $L_k = 1/\gamma_k$ is the Lipschitz constant of $\nabla F_k$,
  we have for any $k=1,2,\dotsc$ that
  \begin{equation}
    F_k(x_{k+1}) \le F_k(x_k) + \langle \nabla F_k(x_k), x_{k+1} - x_k\rangle + \frac{1}{2 \gamma_k} \lVert x_{k+1} - x_k \rVert^2.
  \end{equation}
  By substituting the definition of $x_{k+1}$
  from~\eqref{eq:Fk.update}, we have
  \begin{equation}
    \label{eq:same_k}
    F_k(x_{k+1}) \le F_k(x_k) - \frac{\gamma_k}{2} \Vert \nabla F_k(x_k) \Vert^2.
  \end{equation}
  From Lemma~\ref{lem:relate_function_values}, we have for all $x \in \R^{d}$
  \begin{equation}
    F_{k+1}(x)\le F_k(x) + \frac12(\mu_k - \mu_{k+1})\frac{\mu_k}{\mu_{k+1}}\Vert \nabla g_{\mu_k}(Ax)\Vert^2 \le F_k(x) + (\mu_k - \mu_{k+1}) L_g^2,
  \end{equation}
  where we used in the second inequality that
  $\frac{\mu_k}{\mu_{k+1}} \le 2$. We set $x=x_{k+1}$ and substitute
  into~\eqref{eq:same_k} to obtain
  \begin{equation}
    F_{k+1}(x_{k+1}) \le F_k(x_k) - \frac{\gamma_k}{2} \Vert \nabla F_k(x_k) \Vert^2 + (\mu_k - \mu_{k+1})L_g^2.
  \end{equation}
  By summing both sides of this expression over $k=1,2,\dotsc,K$, and
  telescoping, we deduce that
  \begin{align}
    \nonumber
    \sum_{k=1}^K \frac{\gamma_k}{2} \Vert \nabla F_k(x_k) \Vert^2 & \le F_1(x_1) - F_K(x_K) + (\mu_1 - \mu_K)L_g^2  \\
    \label{eq:telescoping}
    & \le F_1(x_1) - F^* + \mu_1L_g^2.
  \end{align}
  Since
  \[
  \gamma_k = \frac{\mu_k}{\mu_k L_{\nabla h} + \Vert A \Vert^2}\ge
  k^{-1/3}\frac{{(2\rho)}^{-1}}{{(2\rho)}^{-1}L_{\nabla h}+ \Vert A \Vert^2} =
  k^{-1/3}\frac{1}{L_{\nabla h}+ 2 \rho \Vert A \Vert^2}.
  \]
  we have from~\eqref{eq:telescoping} that
  \begin{equation}
    \label{eq:gradient_summable}
    \frac{1}{L_{\nabla h}+ 2 \rho \Vert A \Vert^2} \min_{1 \le j \le K} \,
    \Vert \nabla F_j(x_j) \Vert^2 \frac12 \sum_{k=1}^K
    k^{-1/3} \le F_1(x_1) - F^* + {(2\rho)}^{-1}L_g^2.
  \end{equation}
  Now we observe that
  \begin{equation}
    \label{eq:smoothing-choice1}
    \begin{aligned}
      \sum_{k=1}^{K} k^{-1/3} &\ge \sum_{k=1}^{K} \int_{k}^{k+1} x^{-1/3} \diff x = \int_{1}^{K+1} x^{-1/3} \diff x = \frac32 \left({(K+1)}^{2/3} -1\right) \\
      &\ge {(K+1)}^{2/3} - 1 \ge \frac12 K^{2/3}, \quad K=1,2,\dotsc,
    \end{aligned}
  \end{equation}
  where the final inequality can be checked numerically.  Therefore,
  by substituting into~\eqref{eq:gradient_summable}, we have
  \begin{equation}
    \label{eq:smoothing-choice2}
    \min_{1 \le j \le K} \, \Vert \nabla F_j(x_j) \Vert^2 \le  4\frac{L_{\nabla h}+ (2\rho)\Vert A \Vert^2}{K^{2/3}} \Big(F_1(x_1) - F^* + {(2\rho)}^{-1}L_g^2\Big),
  \end{equation}
  and so
  \begin{equation}
    \label{eq:rate}
    \min_{1 \le j \le K} \, \Vert \nabla F_j(x_j) \Vert \le  \frac{C}{K^{1/3}},
  \end{equation}
  where $C:= 2\sqrt{L_{\nabla h}+ (2\rho)\Vert A \Vert^2} \sqrt{F_1(x_1) - F^* + {(2\rho)}^{-1} L_g^2}$.
  By combining this bound with~\eqref{eq:new_stationarity}, and
  defining $z_j := \prox{\mu_j g}{Ax_j}$, we obtain
  \begin{equation} \label{eq:sj4}
      \min_{1\le j \le k} \, \dist(-\nabla h(x_j), A^*\partial g(z_j) )
      \le \min_{1 \le j \le k} \, \Vert \nabla F_j(x_j) \Vert
      \le \frac{C}{k^{1/3}},
  \end{equation}
  where we deduce from~\eqref{eq:dist_prox_bounded} that
  \begin{equation}
    \label{eq:feasibility-rate}
    \Vert Ax_j - z_j \Vert \le \frac{{(2\rho)}^{-1}L_g}{j^{1/3}} ,
    \quad \mbox{for all $j \ge 1$}.
  \end{equation}
  The second statement concerning surjectivity of $A$ follows from the
  consideration made in~\eqref{eq:sj1}
  to~\eqref{eq:xstar_is_near_stationary}.
\end{proof}

  There is a mismatch between the two bounds in this theorem. The
  first bound (the criticality bound) indicates that during the first
  $k = O(\epsilon^{-3})$ iterations, we will encounter an iteration
  $j$ at which the first-order optimality condition is satisfied to
  within a tolerance of $\epsilon$. However, this bound could have
  been satisfied at an early iteration (that is, $j \ll
  \epsilon^{-3}$), for which value the second (feasiblity) bound, on
  $\Vert Ax_j - \prox{\mu_j g}{Ax_j} \Vert$, may not be particularly
  small. The next section describes an algorithm that remedies this
  defect.

\subsection{An Epoch-Wise Approach with Improved Convergence Guarantees}%
\label{sec:epoch}

 We describe a variant of
 Algorithm~\ref{alg:variable_weakly_smoothing} in which the steps are
 organized into a series of epochs, each of which is twice as long as
 the one before.  We show that there is some iteration $j =
 O(\epsilon^{-3})$ such that both $\Vert Ax_j - \prox{\mu_j g}{Ax_j}
 \Vert $ and $\dist( -\nabla h(x_j), A^*\partial g(\prox{\mu_j g}{Ax_j}))$
 are smaller than the given tolerance $\epsilon$.

\begin{algorithm}[ht!]
  \caption{Variable Smoothing with Epochs}%
  \label{alg:epoch_vs}%
  \begin{algorithmic}
    \REQUIRE $x_1 \in \R^d$ and tolerance $\epsilon>0$;
    \FOR{$l=0,1,\dotsc$}
    \STATE Set $S_l \leftarrow \infty$, Set $j_l \leftarrow 2^l$;
    \FOR{$k=2^l,2^l+1,\dotsc,2^{l+1}-1$}
    \STATE Set $\mu_k \leftarrow {(2\rho)}^{-1}  k^{-1/3}$, define
    $L_k$ as in~\eqref{eq:Lk2}, set $\gamma_k \leftarrow 1/L_k$;
    \STATE Set
    $x_{k+1} \leftarrow x_k - \gamma_k \nabla F_k(x_k)$;
    \IF{$ \Vert \nabla F_{k+1}(x_{k+1})\Vert \le S_l$}
    \STATE Set $S_l \leftarrow \Vert \nabla F_{k+1}(x_{k+1})\Vert$; Set $j_l \leftarrow k+1$;
    \IF{$S_l \le \epsilon$ and $\Vert Ax_{k+1} - \prox{\mu_{k+1} g}{Ax_{k+1}} \Vert  \le \epsilon$}
    \STATE STOP;
    \ENDIF
    \ENDIF
    \ENDFOR
    \ENDFOR
  \end{algorithmic}
\end{algorithm}

\begin{theorem}%
  \label{thm:epoch_var_smoothing}
  Consider solving~\eqref{eq:hg} using Algorithm~\ref{alg:epoch_vs},
  where $h$ and $g$ satisfy the assumptions of
  Theorem~\ref{thm:variable_weakly_smoothing} and $F^*$ defined
  in~\eqref{eq:def.f*} is finite. For a given tolerance $\epsilon>0$,
  Algorithm~\ref{alg:epoch_vs} generates an iterate $x_j$ for some
  $j = O(\epsilon^{-3})$ such that
  \begin{equation}
    \dist( -\nabla h(x_j), A^*\partial g(z_j)) \le \epsilon \quad \text{and} \quad \Vert Ax_j - z_j \Vert \le \epsilon,
  \end{equation}
  where $z_j = \prox{\mu_{j} g}{Ax_{j}}$.
\end{theorem}
\begin{proof}
  As in~\eqref{eq:telescoping}, by using monotonicity of
  $\{ F_k(x_k) \}$ and discarding nonnegative terms, we have that
  \begin{equation}
    \sum_{k=2^l}^{2^{l+1}-1} \frac{\gamma_k}{2} \Vert \nabla F_k(x_k) \Vert^2 \le F_1(x_1) - F^* + {(2\rho)}^{-1}  L_g^2.
  \end{equation}
  With the same arguments as in the earlier proof, we obtain
  \begin{equation}
    \begin{aligned}
      \sum_{k=2^l}^{2^{l+1}-1} k^{-1/3} &\ge \sum_{k=2^l}^{2^{l+1}-1} \int_k^{k+1} x^{-1/3} \diff x = \int_{2^l}^{2^{l+1}} x^{-1/3} \diff x \\
      &= \frac32 \left( {(2^{l+1})}^{2/3} - {(2^{l})}^{2/3} \right)
      = \frac32 \left(2^{2/3}-1\right) {(2^l)}^{2/3} \ge \frac12 {(2^l)}^{2/3}.
    \end{aligned}
  \end{equation}
  Therefore, we have
  \begin{equation}
    \min_{2^l \le j \le 2^{l+1}-1} \, \Vert \nabla F_j(x_j) \Vert \le  \frac{C}{{(2^{l})}^{1/3}},
  \end{equation}
  with $C=2\sqrt{L_{\nabla h}+ 2\rho \Vert A \Vert^2}\sqrt{ F_1(x_1) - F^*
    + {(2\rho)}^{-1} L_g^2}$ as before. Noting that $z_j :=
  \prox{\mu_j g}{Ax_j}$, we have as in~\eqref{eq:sj4} that
  \begin{equation}
    \label{eq:iteration_with_l}
    \min_{2^l \le j \le 2^{l+1}-1} \, \dist( -\nabla h(x_j), A^*\partial g(z_j) ) \le \frac{C}{{(2^l)}^{1/3}},
  \end{equation}
  as previously. Further, we have for $2^l \le j \le 2^{l+1}-1$ that
  \begin{equation}
    \label{eq:feasibility_with_l}
    \Vert Ax_j - z_j \Vert \le L_g \mu  \le \frac{{(2\rho)}^{-1}L_g}{j^{1/3}} \le \frac{{(2\rho)}^{-1}L_g}{{(2^l)}^{1/3}}.
  \end{equation}
  From~\eqref{eq:iteration_with_l} and~\eqref{eq:feasibility_with_l}
  we deduce that Algorithm~\ref{alg:epoch_vs} must terminate before
  the end of epoch $l$, that is, before $2^{l+1}$ iterations have been
  completed, where $l$ is the first nonnegative integer such that
  \begin{equation}
    2^l \ge \max\{C^3, {(2\rho)}^{-3} L_g^3\}\epsilon^{-3}.
  \end{equation}
  Thus, termination occurs after at most
  $2 \max\{C^3, {(2\rho)}^{-3} L_g^3\}\epsilon^{-3}$ iterations.
\end{proof}

For the case of $A$ surjective, we have the following stronger result.

\begin{corollary}%
  \label{cor:}
  Suppose that the assumptions of
  Theorem~\ref{thm:epoch_var_smoothing} hold, that $A$ is also
  surjective, and that the condition $\Vert Ax_{k+1} - \prox{\mu_{k+1}
    g}{Ax_{k+1}} \Vert \le \epsilon$ in Algorithm~\ref{alg:epoch_vs}
  is replaced by $\Vert x_{k+1}-x_{k+1}^{*} \Vert \le \epsilon$,
  where $x_{k+1}^*$ is defined in~\eqref{eq:xjs}.  Then for
  some $j'=O(\epsilon^{-3})$, we have that
  \begin{equation}
    \dist \, ( -\nabla h(x^{*}_{j'}), A^*\partial g(Ax_{j'}^{*})) \le \epsilon
  \end{equation}
  and $\Vert x_{j'}-x_{j'}^{*} \Vert \le \epsilon$.
\end{corollary}
\begin{proof}
  With the considerations made in the previous proof as well as the
  one made in~\eqref{eq:sj1} to~\eqref{eq:xstar_is_near_stationary},
  we can choose $l$ to be the smallest positive integer such that
  \begin{equation}
    2^{l+1} \ge 2\max\{C^{3}, {\sigma_{\min}(A)}^{-3}L_{g}^{3}{(2\rho)}^{-3}\}\epsilon^{-3}.
  \end{equation}
The claim then holds for some $j' \le 2^{l+1}$.
\end{proof}

Although Algorithm~\ref{alg:epoch_vs} seems more complicated than
Algorithm~\ref{alg:variable_weakly_smoothing}, the steps are the
same. The only difference is that for the second algorithm, we do not
search for the iterate that minimizes criticality across \emph{all}
iterations but only across at most the last $k/2$ iterations, where
$k$ is the total number of iterations.

\begin{remark}\label{rem:rate}
  For both versions of our proposed method we use an explicit choice of
  smoothing parameters, choosing $\mu_k$ to be a multiple of
  $\mathcal{O}(k^{-1/3})$. This specific dependence on $k$ achieves a balance
  between criticality and feasibility. As can be seen from~\eqref{eq:rate}
  (criticality measure) and~\eqref{eq:feasibility-rate} (feasibility measure)
  both measures decrease like $k^{-1/3}$. A slower decrease in $\mu_k$ would
  result in a faster decrease in the criticality measure but a slower decrease
  in the feasibility measure---and vice versa.
\end{remark}

\begin{remark}\label{rem:convex}
  Our technique does not adapt in an obvious way to the case
  in which $g$ is actually convex. Typically, we know in advance
  whether or not $h$ and $g$ in~\eqref{eq:hg} are convex, and if
  they are, we could choose one of the well established methods that
  make use of gradients, proximal operators, and possibly
  acceleration. See, for example the proximal accelerated gradient
  approach of~\cite{me-variable-smoothing}, which achieves a rate of
  $\mathcal{O}(k^{-1})$. A method in the spirit
  of~\cite{ghadimi_lan_optimal_rate}, which automatically adapts to
  convexity and simultaneously achieves the optimal rates for both
  nonconvex and convex problems would be desirable, but is outside
  the scope of this work.
\end{remark}

\section{Proximal Gradient}%
\label{sec:prox_grad}

Here we derive a complexity bound for the proximal gradient algorithm
applied to the more elementary problem~\eqref{eq:hg.add} studied in
Section~\ref{sub:stationarity}, that is,
\begin{equation} \label{eq:Fprox}
\min_{x\in\R^{d}} \, F(x) := h(x) + g(x),
\end{equation}
for $h:\R^{d}\to \R$ a $L_{\nabla h}$-smooth function and $g:\R^d \to
\overline{\R}$ a possibly nonsmooth, $\rho$-weakly convex
function. Such a bound has not been made explicit before, to the
authors' knowledge, though it is a fairly straightforward consequence
of existing results. The bound makes an interesting comparison with the
result in Section~\ref{sec:main}, where the nonsmoothness issue
becomes more complicated due to the composition with a linear
operator.  In this section, we assume that a closed-form proximal
operator is available for $g$, and we show that the complexity bound
of $\cO(\epsilon^{-2})$ is the same order as for gradient descent
applied to smooth nonconvex functions.

Standard proximal gradient applied to problem~\eqref{eq:Fprox},
for a given stepsize $\lambda \in (0,\min\{\rho^{-1}/2, L_{\nabla h}^{-1}\}]$
and initial point $x_1$, is as follows:
\begin{align}
   \label{eq:argmin_prox_grad}
  x_{k+1} & := \arg\min_{x\in\R^d} \left\{ g(x) + \langle \nabla h(x_k), x - x_k\rangle +\frac{1}{2 \lambda} \Vert x-x_k \Vert^2 \right\}, \\
  \nonumber
  & = \prox{\lambda g}{x_k - \lambda \nabla h(x_k)}, \quad k=1,2,\dotsc,
\end{align}
where the choice of $\lambda$ ensures that the function to be
minimized in~\eqref{eq:argmin_prox_grad} is
$(\lambda^{-1}-\rho)$-strongly convex, so that $x_{k+1}$ is uniquely defined.

We have the following convergence result.
\begin{theorem}\label{th:pg} Consider the algorithm defined
  by~\eqref{eq:argmin_prox_grad} applied to problem~\eqref{eq:Fprox}, where we
  assume that $g$ is proper, lower semicontinuous and $\rho$-weakly convex and that
  $\nabla h$ is Lipschitz continuous with constant $L_{\nabla h}$. Supposing that
  the stepsize $\lambda \in (0,\min\{\rho^{-1}/2, L_{\nabla h}^{-1}\}]$, we
  have for all $k \ge 1$ that
  \begin{equation} \label{eq:pg.result}
    \min_{2\le j\le k+1} \, \dist (0, \partial(h+g)(x_j)) \le  k^{-1/2} \sqrt{2(F(x_1)-F^*)} \;
    \frac{\lambda^{-1}+L_{\nabla h}}{\sqrt{\lambda^{-1}-\rho}},
  \end{equation}
  where $F^*$ is defined in~\eqref{eq:def.f*}.
\end{theorem}
\begin{proof}
  Note first that the result is vacuous if $F^*=-\infty$, so we assume
  henceforth that $F^*$ is finite.  We have for every $x \in \R^d$
  that
\begin{multline}
  g(x_{k+1}) + h(x_{k}) + \langle \nabla h(x_{k}), x_{k+1}-x_{k}\rangle + \frac{1}{2
    \lambda}\Vert x_{k+1}-x_{k} \Vert^2 + \frac12 (\lambda^{-1} - \rho)
  \Vert x - x_{k+1} \Vert^2\\
  \le g(x) + h(x_k) + \langle \nabla h(x_k), x-x_k\rangle + \frac{1}{2 \lambda} \Vert x - x_k \Vert^2.
\end{multline}
By applying the inequality
\begin{equation}
  h(x_{k+1}) \le h(x_{k}) + \langle \nabla h(x_{k}), x_{k+1}-x_{k}\rangle +
  \frac{1}{2 \lambda}\Vert x_{k+1}-x_{k} \Vert^2, \quad \mbox{for all
    $x \in \R^d$},
\end{equation}
obtained from the Lipschitz continuity of $\nabla h$ and the fact that
$\lambda \le L_{\nabla h}^{-1}$, we deduce that
\begin{equation}
  F(x_{k+1}) + \frac12 (\lambda^{-1} - \rho) \Vert x - x_{k+1} \Vert^2 \le g(x) + h(x_k) + \langle \nabla h(x_k), x-x_k\rangle + \frac{1}{2 \lambda} \Vert x - x_k \Vert^2,
\end{equation}
for every $x \in \R^d$.  By setting $x=x_k$, we obtain
\begin{equation}
  F(x_{k+1}) + \frac12 (\lambda^{-1} - \rho) \Vert x_k - x_{k+1} \Vert^2 \le F(x_k),
\end{equation}
which shows, together with the definition~\eqref{eq:def.f*}, that
\begin{equation} \label{eq:su8}
\sum_{k=1}^\infty \Vert x_k - x_{k+1} \Vert^2 \le \frac{2(F(x_1)-F^*)}{\lambda^{-1}-\rho}.
\end{equation}
From the optimality conditions for~\eqref{eq:argmin_prox_grad}, we
obtain
\begin{equation}
  0 \in \nabla h(x_k) + \partial g(x_{k+1}) + \lambda^{-1}(x_{k+1}-x_k)
\end{equation}
which also shows that
\begin{equation} \label{eq:su9}
  w_{k+1} := \frac{1}{\lambda} (x_k - x_{k+1}) + \nabla h(x_{k+1}) - \nabla h(x_k) \in \partial(h+g)(x_{k+1}),
\end{equation}
so that
\[
  \| w_{k+1} \|^2 \le {(\lambda^{-1}+L_{\nabla h})}^2 \| x_k-x_{k+1}\|^2.
\]
By combining this bound with~\eqref{eq:su8}, we obtain
\[
  \sum_{k=1}^{\infty} \|w_{k+1}\|^2 \le 2(F(x_1)-F^*) \frac{{(\lambda^{-1}+L_{\nabla h})}^2}{\lambda^{-1}-\rho}.
\]
from which it follows that
\[
  \min_{1\le j\le k} \| w_{j+1} \| \le \sqrt{2 (F(x_1)-F^*)} \; \frac{(\lambda^{-1}+L_{\nabla h})}{\sqrt{\lambda^{-1}-\rho}}.
\]
The result now follows from~\eqref{eq:su9}, when we note that
\[
  \min_{1\le j\le k} \, \dist (0, \partial(h+g)(x_{j+1})) \le \min_{1\le j\le k} \, \Vert w_{j+1} \Vert.
\]
\end{proof}

This theorem indicates that the proximal gradient algorithm requires
at most $\cO(\eps^{-2})$ to find an iterate with $\eps$-approximate
stationarity. This bound contrasts with the bound $\cO(\eps^{-3})$ of
Section~\ref{sec:main} for the case of general $A$. Moreover, the
$\cO(\eps^{-2})$ bound has the same order as the bound for gradient
descent applied to general smooth nonconvex optimization.



\section{Conclusions}%
\label{sec:}

We consider a standard problem formulation in which a linear
transformation of the input variables is composed with a nonsmooth
regularizer and added to a smooth function. In most works, the
regularizer is assumed to be convex, but we extend here to the case
in which it is only weakly convex.  This extension allows for
functions which introduce desired properties, such as sparsity,
without causing a bias. (Two examples from robust statistics are
minimax concave penalty (MCP) and smoothly clipped absolute
deviation (SCAD).) We propose a novel method based on the variable
smoothing framework and show a complexity of
$\mathcal{O}(\epsilon^{-3})$ to obtain an $\epsilon$-approximate
solution.  This iteration complexity falls strictly between the
iteration complexity of smooth (nonconvex) problems
($\mathcal{O}(\epsilon^{-2})$) and that of the black box subgradient
method for weakly convex function ($\mathcal{O}(\epsilon^{-4})$)
which assumes no knowledge of the structure of the nonsmoothness.

A performance comparison between our smoothed approach and the
black-box subgradient algorithm on an image denoising problem that
uses an MCP total variation regularizer is shown in
Figure~\ref{fig:images-used-denoising}.

\begin{figure}
  \centering
  \begin{subfigure}[b]{0.32\linewidth}
    \centering
    \includegraphics[width=\linewidth]{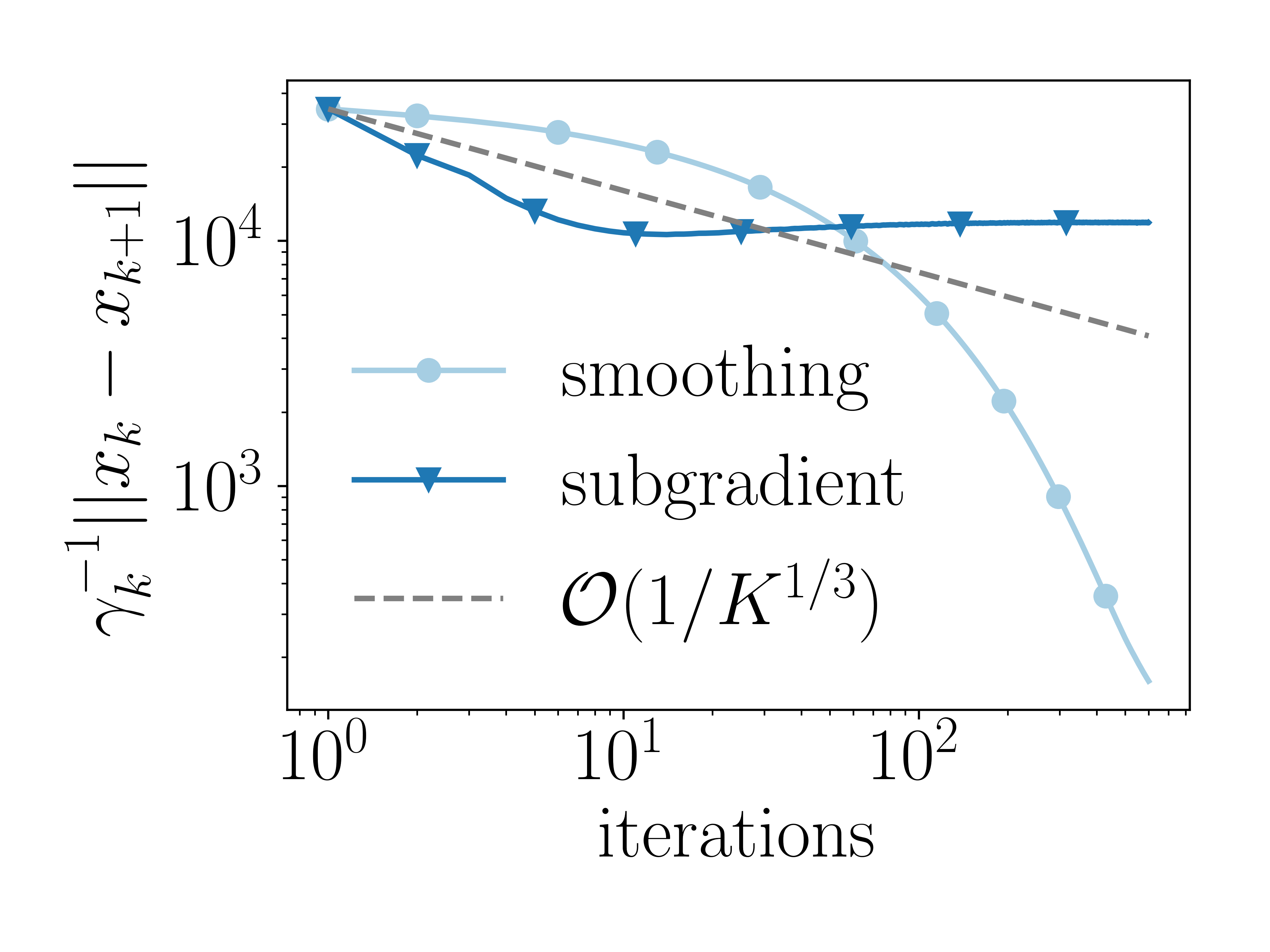}
    \caption{Norm of (sub)gradient.}
  \end{subfigure}
  \begin{subfigure}[b]{0.32\linewidth}
    \centering
    \includegraphics[width=\linewidth]{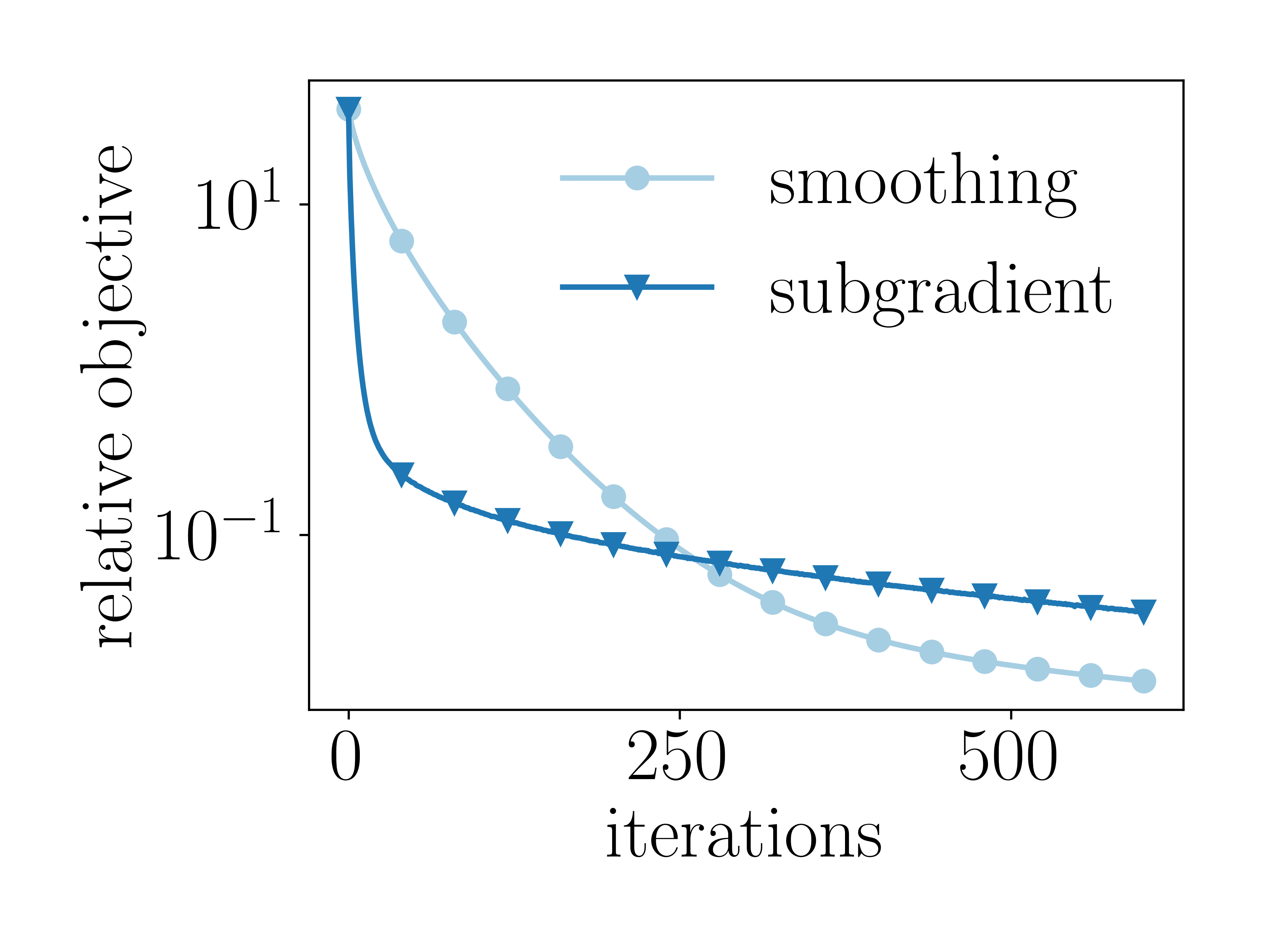}
    \caption{Objective function.}
  \end{subfigure}
  \begin{subfigure}[b]{0.32\linewidth}
    \centering
    \includegraphics[width=\linewidth]{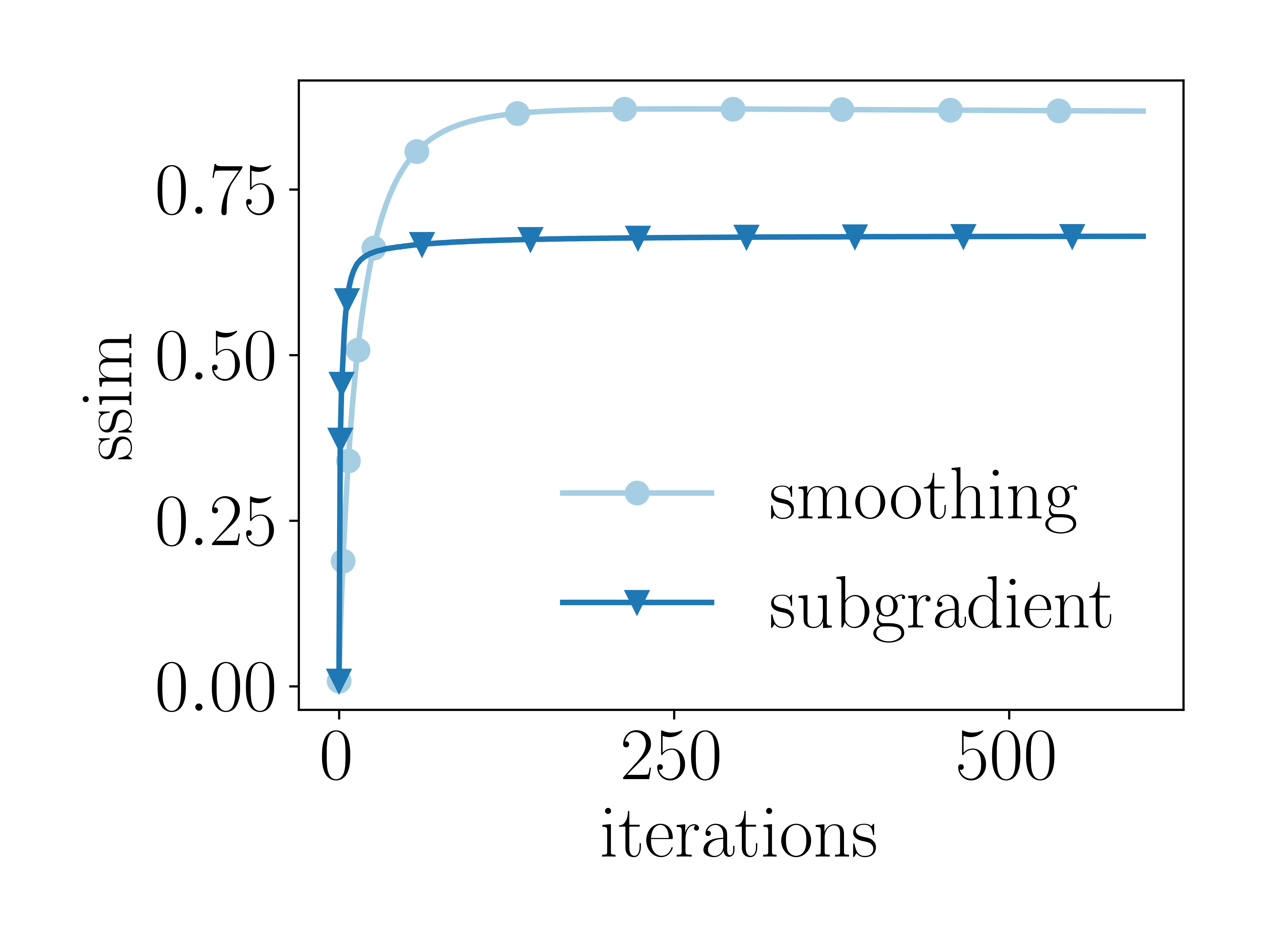}
    \caption{Reconstruction.}
  \end{subfigure}
  \caption{Progress of our smoothing algorithm and a naive subgradient algorithm on an image denoising problem, where minimax concave penalty is used instead of the $1$-norm in the anisotropic total variation, showing better performance by the smoothing approach.  \textbf{Left:} The difference of consecutive iterates scaled by the inverse of the stepsize, representing the norm of the (sub)gradient used at each iteration. \textbf{Middle:} Relative difference between the objective function at the current iterate and an approximate minimum. \textbf{Right:} The quality of the resulting reconstruction measured via the \emph{structural similarity index measure}, see~\cite{ssim}.}%
  \label{fig:images-used-denoising}
\end{figure}

\paragraph{Acknowledgements}
  Research of the first author was supported by the doctoral programme
  \textit{Vienna Graduate School on Computational Optimization (VGSCO)}, FWF
  (Austrian Science Fund), project W 1260. Research of the second author was
  supported by NSF Awards 1628384, 1634597, and 1740707; Subcontract 8F-30039
  from Argonne National Laboratory; and Award N660011824020 from the DARPA
  Lagrange Program.

\bibliographystyle{abbrv}
\bibliography{bibfile}

\end{document}